\documentclass[english]{article}
\usepackage[T1]{fontenc}
\usepackage[utf8x]{inputenc}
\usepackage{graphicx}
\usepackage{amssymb}
\usepackage{amsmath}
\usepackage{amsthm}
\usepackage{color}
\usepackage{amsfonts}
\usepackage{mathrsfs}
\usepackage{latexsym}
\usepackage{geometry} 
\usepackage{textcomp}
  \usepackage[english]{babel}
\usepackage{tikz}
\usetikzlibrary{matrix,arrows}

\DeclareMathOperator{\Hdg}{Hdg}

\theoremstyle{definition} 
\newtheorem{defi}{Definition}[section]

\theoremstyle{plain} 
\newtheorem*{bigtheo}{Theorem}
\newtheorem{theo}[defi]{Theorem}
\newtheorem{lemm}[defi]{Lemma}  
\newtheorem{prop}[defi]{Proposition}
\newtheorem{hypo}[defi]{Hypothesis}

\theoremstyle{remark} 
\newtheorem{rema}[defi]{Remark}

\begin{document}

\title{Hodge stratification in low dimension}
\author{Stéphane Bijakowski}
\date{}
\maketitle

\begin{abstract}
We define and study the Hodge stratification for the special fiber of Shimura varieties defined with the Pappas-Rapoport condition, in the case of low ramification index ($e \leq 3$). For $e \leq 2$, the Hodge polygon induces a strong stratification. For $e=3$, one needs to introduce several polygons. They describe the isomorphism class of the sheaf of differentials with extra structure, and induce a strong stratification on the variety. 
\end{abstract}

\section*{Introduction}

Let $p$ be an odd prime, and $X$ be the special fiber of the modular curve. Considering the structure of the universal elliptic curve at $p$, one is led to consider two possibilities: either it is ordinary at $p$, or supersingular. The ordinariness can be seen thanks to the number of points of the $p$-torsion, the structure of the $p$-divisible group (which is split, with a multiplicative part, and an étale part), or the Hasse invariant (which is non-zero). One then has a stratification on the variety, with the ordinary locus, and supersingular points. \\
$ $\\
For more general varieties, for example Siegel or Hilbert-Siegel varieties, one can define different stratifications. Assume that the prime $p$ is unramified in the datum, and consider the special fiber of such a variety. One can then look at the $p$-rank of the abelian scheme, which gives a stratification indexed by an integer. A finer stratification is given by the isomorphism class of the $p$-torsion of the abelian scheme, the Ekedahl-Oort stratification (see \cite{Oo}). Another possibility is to consider the associated $p$-divisible group, up to isogeny, which gives the Newton stratification (see \cite{VW} for example). \\
$ $\\
The geometry of these varieties have been extensively studied in the unramified case. However, few results are known when the prime $p$ ramifies. One is led to consider Shimura varieties defined by Pappas and Rapoport (\cite{P-R}, \cite{P-R2}), whose integral models are smooth (\cite{BH_PR}). In \cite{BH_ram}, several polygons are defined: the Newton polygon, the Hodge polygon, and the PR polygon (which is constant on the variety). Contrary to the unramified case, the Hodge polygon can vary on the variety, and one can try to use this polygon to define a stratification. \\
$ $\\
Let us consider the Hilbert-Siegel variety $X_0$ associated to a totally real field $F_0$ (we also consider unitary Shimura varieties), defined with the Pappas-Rapoport models, and $X$ its special fiber. Let $p$ be a prime, and assume for simplicity that $p$ is totally ramified in $F_0$, with degree $e$ and uniformizer $\pi$. If $\omega$ denotes the sheaf of differentials, one has a filtration
$$0 \subseteq \omega_1 \subseteq \dots \subseteq \omega_e = \omega$$

If $x \in X(k)$, the Hodge polygon at $x$ describes the structure of $\omega$ as a $k[T]/T^e$ (with $T$ acting by $\pi$). In general, the Hodge polygon does not induce a strong stratification. One has the following theorem.

\begin{bigtheo}
If $e \leq 2$, the Hodge polygon induces a strong stratification on the variety. Assume that $e=3$; the isomorphism class of $\omega$ over a point in $X$ is described by the three polygons $\Hdg(\omega), \Hdg(\omega_2), \Hdg(\omega / \omega_1)$. Moreover, these three polygons define a strong stratification on $X$.
\end{bigtheo}

If $M$ is a $k[T]/T^i$-module (for a field $k$ and integer $i$), then $\Hdg(M)$ is the polygon describing its structure as a $k[T]/T^i$-module. By a strong stratification, we mean that the closure of a stratum is equal to a union of other strata. \\
We also give an explicit description of the possible values of the three polygons. \\
Note that one can make a link with Spaltenstein varieties (\cite{Sp}): in these varieties, one fix the structure of $\omega$, and consider the different possibilities for the filtration. The result presented here is then more general, since we allow the structure of $\omega$ to vary. \\
$ $\\
Let us now talk about the difficulties when $e \geq 4$. First of all, it is not true in general that the isomorphism class of the filtration $(\omega_i)$ gives a finite number of strata. Indeed, one can consider the case where the multiplication by $\pi^k$ is an isomorphism between $\omega_{2(k+1)} / \omega_{2k}$ and $\omega_2$. The space $\omega_{2k+1}$ is then determined by a subspace inside $\omega_2$. But classifying a large number of subspaces inside a fixed vector space may give an infinite number of possibilities. One can also look at the appendix of \cite{A-G} for an example of an infinite number of isomorphism classes. \\
$ $\\
One could try to generalize the above theorem, considering in the general case all the polygons $\Hdg( \omega_i / \omega_j)$. The issue is that these polygons only give information about the dimension of spaces of the form $\pi^{-k} \omega_j \cap \omega_i$. When $e \geq 4$, the space $\omega_4 \cap \pi^{-1} \omega_2 \cap \pi^{-2} \omega_0$ is not of this form, and its dimension cannot be deduced from the previous polygons. \\
$ $\\
Let us now talk about the organization of the article. In the first section, we introduce the set of polygons that we will consider. In section $2$, we show how these polygons can be applied to the study of $k[T]/T^e$-modules, and we specialize to the case $e=3$ in section $3$. In section $4$, we prove the result concerning the stratification of Shimura varieties. \\
$ $\\
The author is part of the project ANR-19-CE40-0015 COLOSS.

\section{Polygons}

\subsection{Definition}

Let $h \geq 1$ be an integer. Les $N \geq 1$, and $d_1, \dots, d_N$ be integers between $0$ and $h$.

\begin{defi}
We define the polygon $P(d_1, \dots, d_N)$ by the formula
$$P(d_1, \dots, d_N) (x) = \frac{1}{N} \sum_{i=1}^N \max(0, x + d_i - h)$$
for all real $x$ between $0$ et $h$.
\end{defi} 

This polygon is convex, its breakpoints have $x$-coordinates in $\mathbb{Z}$. The order of the integers $d_i$ is irrelevant : one has $P(d_1, \dots, d_N) = P(d_{\sigma(1)}, \dots, d_{\sigma(N)})$ for every permutation $\sigma$. \\
They are useful to describe a certain type of polygons.

\begin{defi}
Let $N \geq 1$ be an integer; let us define $\mathcal{P}_N$ to be the set of convex polygons between $0$ and $h$, whose breakpoints have integral $x$-coordinates and whose slopes are in $\frac{1}{N} \mathbb{Z} \cap [0,1]$.
\end{defi}

\begin{prop}
Let $N \geq 1$. The polygons $P(d_1, \dots, d_N)$ are in $\mathcal{P}_N$. \\
Let $P \in \mathcal{P}_N$; there exist a unique collection of integers $d_1, \dots, d_N$ (up to a permutation) such that $P = P(d_1, \dots, d_N)$. 
\end{prop}

\begin{proof}
Let us consider the polygon $P(d_1, \dots, d_N)$. One can assume that the integers are ordered in such a way that $d_1 \geq \dots \geq  d_N$. The slopes of this polygon are obviously in $\frac{1}{N} \mathbb{Z} \cap [0,1]$. More precisely, this polygon has slope $i/N$ with multiplicity $d_i - d_{i+1} $, for every $0 \leq i \leq N$ (with $d_0 =h$ and $d_{N+1} = 0$ by convention). \\
Now let $P \in \mathcal{P}_N$. It has slope $i/N$ with multiplicity $a_i$, for every $0 \leq i \leq N$. One must have $a_0 + \dots + a_N = h$. One thus sees that there exists a unique collection of integers $d_1 \geq \dots \geq d_N$ such that $P = P(d_1, \dots, d_N)$ given by the formula 
$$d_i = h - (a_0 + \dots + a_{i-1})$$
for $1 \leq i \leq N$.

\end{proof}

\begin{rema}
If $N,k$ are integers, there is a natural inclusion $\mathcal{P}_N \subseteq \mathcal{P}_{kN}$. If $P$ is the polygon $P(d_1, \dots, d_N)$, this operation consists in writing each integer $d_i$ with multiplicity $k$ (i.e. $P=P(d_1, \dots, d_1, d_2, \dots, d_2, \dots, d_N, \dots, d_N)$).
\end{rema}

This description allows us to define the following operation.

\begin{defi}
Let $N_1, N_2$ be two integers. To $P_1 \in \mathcal{P}_{N_1}$ and $P_2 \in \mathcal{P}_{N_2}$, one can attach the polygon $P_1 \star P_2 \in \mathcal{P}_{N_1 + N_2}$ by the following formula
$$P_1 \star P_2 (x) = \frac{1}{N_1 + N_2} (N_1 P_1(x) + N_2 P_2(x))$$
\end{defi}

Note that if $P_1 = P(d_1, \dots, d_{N_1})$ and $P_2 = P(d_1', \dots, d_{N_2}')$, then
$$P_1 \star P_2 := P(d_1, \dots, d_{N_1}, d_1', \dots, d_{N_2}')$$ 

If $P_1$ and $P_2$ are two polygons, we say that $P_1$ lies above $P_2$, and write $P_1 \geq P_2$ if $P_1 (x) \geq P_2(x)$ for all real $0 \leq x \leq h$.

\begin{prop}
Let $d_1 \geq \dots \geq d_N$ and $d_1' \geq \dots \geq d_N'$. Then $P(d_1, \dots, d_N) \geq P(d_1', \dots, d_N')$ if and only if
$$d_1'+ \dots + d_i' \leq d_1 + \dots + d_i$$
for all $1 \leq i \leq N$.
\end{prop}

\begin{proof}
Let $P_1 = P(d_1, \dots, d_N)$ and $P_2 = P(d_1', \dots, d_N')$ and assume that $P_1 \geq P_2$, and let $1 \leq i \leq N$ be an integer. Then $N P_1 (h-d_i) = d_1 + \dots + d_{i-1} - (i-1) d_i$. Moreover
$$N P_2 (h-d_i) = \sum_{j=1}^N \max(0, d_j'-d_i) \geq \sum_{j=1}^i (d_j' - d_i) = d_1' + \dots + d_i' - i d_i  $$
The condition $P_1 (h-d_i) \geq P_2 (h-d_i)$ thus implies that $d_1 + \dots + d_i \geq d_1' + \dots + d_i'$. \\
Conversely assume that $d_1'+ \dots + d_i' \leq d_1 + \dots + d_i$ for all $1 \leq i \leq N$. Since the polygons are both convex, and the breakpoints of $P_2$ have $x$-coordinates $h-d_i'$, it is enough to prove that $P_1 (h-d_i')\geq P_2 (h-d_i')$ for $1 \leq i \leq N+1$ (with $d_{N+1}' = 0$). Let $1 \leq i \leq N+1$; one has $N P_2(h-d_i') = d_1' + \dots + d_{i-1}' - (i-1) d_{i}'$. Then
$$N P_1 (h-d_i') = \sum_{j=1}^N \max(0, d_j - d_i') \geq \sum_{j=1}^i (d_j - d_i') = (d_1 + \dots + d_i) - i d_i' \geq (d_1' + \dots + d_i') - i d_i'$$
which yields the result.
\end{proof}

\section{PR datum}

Let $k$ be a field, $e \geq 1$ an integer, and $M$ be a finite vector space over $k$ with a $k[T] /T^e$-action. Assume that $M$ is generated by at most $h$ elements as a $k[T]/T^e$-module. One can then write

$$ M \simeq \oplus_{i=1}^h k[T] / T^{a_i}$$
for some integers $a_1, \dots, a_h$ between $0$ and $e$.

\begin{defi}
We define the Hodge polygon $\Hdg(M)$ of $M$ as the polygon with slopes $\frac{a_1}{e}, \dots, \frac{a_h}{e}$.
\end{defi}

One can look at \cite{BH_ram} section 1 for more details about the Hodge polygon. \\
This polygon belongs to $\mathcal{P}_e$. Let $\delta_i := \dim M[T^i] / M[T^{i-1}]$ for $1 \leq i \leq e$. Since the multplication by $T$ induces an injection from $M[T^{i+1}] / M[T^i]$ to $M[T^i] / M[T^{i-1}]$, one has the inequalities $\delta_1 \geq \delta_2 \geq \dots \geq \delta_e$.

\begin{prop}
One has $\Hdg(M) = P(\delta_1, \dots, \delta_e)$.
\end{prop}

\begin{proof}
From the definition, one finds that the quantity $d_i - d_{i+1}$ is equal to the number of integers $a_j$ equal to $i$. The polygon $\Hdg(M)$ has thus slope $i/e$ with multiplicity $d_i - d_{i+1}$, and is equal to $P(d_1, \dots, d_e)$.
\end{proof}

\begin{defi}
Let $\mu= (d_1, \dots, d_e)$. A PR datum of type $\mu$ for $M$ is a filtration 
$$0=M_0 \subseteq M_1 \subseteq \dots \subseteq M_e = M$$
such that
\begin{itemize}
\item $M_i$ is a vector subspace of $M$.
\item $ T \cdot M_i \subseteq M_{i-1}$ for $1 \leq i \leq e$.
\item The dimension of $M_i / M_{i-1}$ is equal to $d_i$ for $1 \leq i \leq e$.
\end{itemize}
\end{defi}

The terminology comes from the definition by Pappas and Rapoport of the special fiber of Shimura varieties (see \cite{P-R}).

\begin{prop}
Let $\mu= (d_1, \dots, d_e)$, $\sigma$ a permutation and $\mu' = (d_{\sigma(1)}, \dots, d_{\sigma(e)})$. Then there exists a PR datum of type $\mu$ for $M$ if and only if there exists a PR datum of type $\mu'$ for $M$.
\end{prop}

\begin{proof}
It is enough to prove the result when $\sigma$ is a transposition exchanging two consecutive integers, say $i$ and $i+1$. Assume that $M_0 \subseteq \dots \subseteq M_e$ is a PR datum of type $\mu$ for $M$. Let $N := M_{i+1} / M_{i-1}$, and $N_0 := M_i / M_{i-1}$. Then $N$ is a vector space of dimension $d_i + d_{i+1}$, with an action of $k[T] / T^2$. The vector subspace $N_0$ has dimension $d_i$, and one has
$$T \cdot N \subseteq N_0 \subseteq N[T]$$
Let $r$ be the dimension of $T \cdot N$; the dimension of $N[T]$ is then $d_i + d_{i+1} -r$, and one finds that $r \leq \min(d_i, d_{i+1})$. This means that it is possible to find a vector space $N_1$ of dimension $d_{i+1}$ such that
$$T \cdot N \subseteq N_1 \subseteq N[T]$$
The space $N_1$ gives a PR datum of type $\mu'$ for $M$.
\end{proof}

\begin{theo} \label{HdgPR}
Let $\mu= (d_1, \dots, d_e)$. There exists a PR datum of type $\mu$ if and only if
$$\Hdg(M) \geq P(d_1, \dots, d_e)$$
\end{theo}

\begin{proof}
From the previous proposition, one can assume that $d_1 \geq \dots \geq d_e$. Assume that there exists a PR datum of type $\mu$ written $M_0 \subseteq \dots M_e$. Since $M_i \subseteq M[T^i]$, one must have $d_1 + \dots + d_i \leq \delta_1 + \dots + \delta_i$, hence the result. \\
Now assume that $\Hdg(M) \geq P(d_1, \dots, d_e)$. Let us write $\Hdg(M) = P(\delta_1, \dots, \delta_e)$ with $\delta_1 \geq \dots \geq \delta_e$. We see in particular that $\dim M[T^i] = \delta_1 + \dots + \delta_i$ and $\dim T^i M = \delta_{i+1} + \dots + \delta_e$. The fact that $\Hdg(M) \geq P(d_1, \dots, d_e)$ implies the inequalities
$$d_1 + \dots + d_i \leq \delta_1 + \dots + \delta_i \qquad \delta_e + \dots + \delta_{e-i} \leq d_e + \dots + d_{e-i}$$
The second inequality follows from the relation $d_1 + \dots + d_e = \delta_1 + \dots + \delta_e$. \\
We construt the PR datum of type $\mu$ in the following way. We define $M_1$ of dimension $d_1$ inside $M[T]$ such that the dimension of $M_1 \cap T^j M$ is maximal. Note that this is possible since the dimension of $M[T]$ is equal to $\delta_1 \geq d_1$. One then obtains that the dimension of $T^j M \cap M_1$ is equal to $\alpha_1^j := \min (\delta_{j+1}, d_1)$. Since $\delta_e \leq d_e \leq d_1$, one has $\alpha_1^{e-1} = \delta_e$, and $M_1$ contains $T^{e-1} M$. \\
One then considers the vector space $T^{-1} M_1$; it has dimension 
$$\delta_1 + \alpha_1^2 = \min (\delta_1 + \delta_2, \delta_1 + d_1)$$
One defines $M_2$ as a vector space containing $M_1$ inside $T^{-1} M_1$. This is allowed, since $M_2$ must have dimension $d_1 + d_2$, and that $d_1 + d_2 \leq  \min (\delta_1 + \delta_2, \delta_1 + d_1)$ (since $d_2 \leq d_1 \leq \delta_1$, one has indeed $d_1 + d_2 \leq d_1 + \delta_1)$. One constructs $M_2$ in such a way that the dimension of $M_2 \cap T^j M$ is maximal. Since the dimension of $T^j M \cap T^{-1} M_1$ is equal to $\alpha_1^{j+1} + \delta_{j+1}$, the dimension of $M_2 \cap T^j M$ is equal to 
$$\alpha_2^{j} := \min (d_2 + \alpha_1^j, \alpha_1^{j+1} + \delta_{j+1}) = \min ( d_1 + d_2, \delta_{j+1} + d_2, d_1 + \delta_{j+1}, \delta_{j+2} + \delta_{j+1}) = \min ( d_1 + d_2, \delta_{j+1} + d_2, \delta_{j+2} + \delta_{j+1}) $$ 
Note that $\alpha_2^{e-2} = \min (d_1 + d_2, d_2 + \delta_{e-1}, \delta_e + \delta_{e-1}) = \delta_e + \delta_{e-1}$. Indeed, one has $\delta_e \leq d_e \leq d_2$, and $\delta_e + \delta_{e-1} \leq d_e + d_{e-1} \leq d_1 + d_2$. This means that $T^{e-2} M \subset M_2$. \\
One then constructs the vector spaces $M_1 \subseteq \dots \subseteq M_e = M$ such that $M_i \subseteq T^{-1} M_{i-1}$, and the dimension of $T^j \cap M_i$ is equal to 
$$\alpha_i^{j} := \min ( d_1 + \dots + d_i, \delta_{j+1} + d_2 + \dots +d_i, \dots,  \delta_{j+1} + \dots + \delta_{j+i}) $$ 
Indeed, assume the spaces $M_1, \dots, M_i$ satisfy the required property. One constructs $M_{i+1}$ inside $T^{-1} M_i$, such that the dimension of $M_{i+1} \cap T^j M$ is maximal. Note that the dimension of $T^{-1} M_i$ is equal to 
$$\delta_1 + \alpha_i^{1} = \min ( \delta_1 + d_1 + \dots + d_i, \delta_1+ \delta_{2} + d_2 + \dots +d_i, \dots,  \delta_1 + \delta_{2} + \dots + \delta_{i+1})$$
and this quantity is greater or equal than $d_1 + \dots + d_{i+1}$. The dimension of $M_{i+1} \cap T^j M$ is then equal to
$$\min (\alpha_i^{j} + d_{i+1}  , \alpha_i^{j+1} + \delta_{j+1} ) = \alpha_{i+1}^{j}$$
This gives the result by induction on $i$. Note that the dimension of $M_i \cap T^{e-i} M$ is equal to
$$\alpha_i^{e-i} = \min ( d_1 + \dots + d_i, \delta_{e-i+1} + d_2 + \dots +d_i, \dots,  \delta_{e-i+1} + \dots + \delta_{e}) = \delta_{e-i+1} + \dots + \delta_{e} $$
which proves that $T^{e-i} M \subseteq M_i$. In particular, one sees that $T M \subseteq M_{e-1}$, and the filtration $M_1 \subseteq \dots \subseteq M_e = M$ is indeed a PR datum of type $\mu$.
\end{proof}

\begin{rema}
The above result gives a simpler proof of \cite{BH_ram} Th. 1.3.1, not using exterior algebras.
\end{rema}

\begin{prop} \label{Hdgfilt}
Let $M$ be a $k[T]/T^e -module$ and $1 \leq i \leq e$. Assume that there exists a filtration 
$$0 \subseteq N \subseteq M$$
such that $T^i N = 0$, $T^{e-i} M \subseteq N$. Then $\Hdg(M) \geq \Hdg(N) \star \Hdg(M/N)$.
\end{prop}

In the above proposition $N$ and $M/N$ are respectively $k[T]/T^i$ and $k[T]/T^{e-i}$-modules.

\begin{proof}
Let $\delta_1 \geq \dots \geq \delta_i$ be the elements such that $\Hdg(N) = P(\delta_1, \dots, \delta_i)$. Similarly, let $\delta_{i+1} \geq \dots \geq \delta_e$ be the elements such that $\Hdg(N) = P(\delta_{i+1}, \dots, \delta_e)$. The module $N$ thus satisfies a PR datum of type $(\delta_1, \dots,  \delta_i)$. The module $M/N$ thus satisfies a PR datum of type $(\delta_{i+1}, \dots,  \delta_e)$. The module $M$ thus satisfies a PR datum of type $(\delta_1, \dots, \delta_e)$. Hence the result.
\end{proof}

\section{Case $e=3$}

Let $k$ be a field, and let $d_1 \geq d_2 \geq d_3$ be integers between $0$ and $h$. Let $\mu =(d_1, d_2, d_3)$.

\begin{defi}
Define $X$ be the set of isomorphisms classes of $k[T]/T^3$-modules $M$ with a PR datum of type $\mu$.
\end{defi} 

\begin{defi}
Let $Y$ be the set consisting of tuples of polygons $(P_1, P_2, P_3)$ such that
\begin{itemize}
\item $P_1 \in \mathcal{P}_3$, $P_2, P_3 \in \mathcal{P}_2$
\item $P_1(h) = d_1+d_2+d_3$, $P_2(h) = d_1+d_2$ and $P_3(h) = d_2+d_3$
\item $P_1 \geq P_2 \star P(d_3)$ and $P_1 \geq P_3 \star P(d_1)$.
\item $P_2 \geq P(d_1, d_2)$ and $P_3 \geq P(d_2, d_3)$
\end{itemize}
\end{defi}

\begin{prop}
There exists a map $\phi : X \to Y$ defined by
$$\phi(0 \subseteq M_1 \subseteq M_2 \subseteq M_3 = M) = (\Hdg(M), \Hdg(M_2), \Hdg(M/M_1))$$
\end{prop}

\begin{proof}
One has to check that the polygons $(\Hdg(M), \Hdg(M_2), \Hdg(M/M_1))$ satisfy the required conditions. The first and second properties are obviously satisfied. The third one follows from \ref{Hdgfilt}. The last one follows from \ref{HdgPR}. \\
Lastly, it is clear that the element $\phi(0 \subseteq M_1 \subseteq M_2 \subseteq M_3 = M) $ depends only on the isomorphism class in $X$.
\end{proof}

\begin{defi}
Define $Y^{adm}$ as the subset of $Y$ consisting of the points $(P_1, P_2, P_3)$ satisfying
$$\delta_1 + \max(d_2, \delta_2) \leq \alpha_1 + \beta_1 $$
with $P_1 = P(\delta_1, \delta_2, \delta_3)$, $P_2 = P(\alpha_1, \alpha_2)$, $P_3 = P(\beta_1, \beta_2)$ ($\delta_1 \geq \delta_2 \geq \delta_3$, $\alpha_1 \geq \alpha_2$, $\beta_1 \geq \beta_2$).
\end{defi}

\begin{theo}
The map $\phi$ induces a bijection between $X$ and $Y^{adm}$.
\end{theo}

\begin{proof}
Let us prove the injectivity. We keep the notation from the previous definition. We then have
$$ M \simeq (k[T]/T^3)^{\delta_3} \oplus (k[T]/T^2)^{\delta_2 - \delta_3} \oplus (k[T]/T)^{\delta_1- \delta_2} $$
This gives the structure of $M$ up to isomorphism. Define the free $k[T]/T^3$-module $H := (k[T] /T^3)^{\delta_1}$, with basis $e_1, \dots, e_{\delta_1}$. Then $M$ is isomorphic to the submodule of $H$ generated by
$$e_1, \dots, e_{\delta_3}, Te_{\delta_3 + 1}, \dots, Te_{\delta_2}, T^2 e_{\delta_2+1}, \dots, T^2 e_{\delta_1}$$
One has the inclusion 
$$T^2 M \subseteq T M_2 \subseteq TM \cap M_1 \subset M_1 \cap M[T]$$
The dimensions of these spaces are $\delta_3, \alpha_2, d_1 + \beta_1 - \delta_1, d_1, \delta_1$. One can then up to a change of basis, assume that $M_1$ has basis
$$T^2 e_1, \dots, T^2 e_{d_1 + \beta_1 - \delta_1}, T^2 e_{\delta_2 + 1}, \dots, T^2 e_{\delta_2 + \delta_1 - \beta_1}$$
the space $T M_2$ having $e_1, \dots, e_{\alpha_2}$ as basis. \\
Up to a change of basis, one can moreover assume that $M_2$ has basis over $k$ given by
$$T e_1, \dots, T e_{\alpha_2}, T^2 e_1, \dots, T^2 e_{\alpha_1}$$
This proves that given an element $y \in Y$, if there exists a module $M$ mapping to $y$, it is uniquely determined up to an isomorphism, hence the injectivity of $\phi$. \\
Let us now prove the image of $\phi$ is exactly $Y^{adm}$. \\
Let $\delta_1 \geq \delta_2 \geq \delta_3$, with $\delta_1 + \delta_2 + \delta_3 = d_1 + d_2 + d_3$, $P(\delta_1, \delta_2, \delta_3) \geq P(d_1, d_2, d_3)$ and define the modules $M$ as before. We will see what are the possible structures of $PR$ datum that one can impose on $M$. \\
First of all, one must have $T^2 M \subseteq M_1 \subseteq M[T]$. Note that one has indeed $\delta_3 \leq d_1 \leq \delta_1$. Then one gets the condition
$$ \max(\delta_3, d_1 + \delta_2  - \delta_1) \leq \dim (M_1 \cap TM) \leq \min (d_1, \delta_2)$$
If $\Hdg(M/M_1) = (\beta_1, \beta_2)$, then $\dim (M_1 \cap TM) = \beta_1 + d_1 - \delta_1$. The inequality $P(\delta_1, \delta_2, \delta_3) \geq P(\beta_1, \beta_2, d_1)$ implies automatically the inequality $\dim (M_1 \cap TM) \leq \min (d_1, \delta_2)$. The other inequality gives the conditions $\beta_2 \leq \delta_2 \leq \beta_1$. \\
Now assume that $M_1$ is constructed. One must construct $M_2$ with the conditions
$$TM + M_1 \subseteq M_2 \subseteq T^{-1} M_1$$
These vector spaces have dimension $\delta_1 + \delta_2 + \delta_3 - \beta_1$, $d_1 + d_2$ and $\beta_1 + d_1$ respectively. The relation $P(\beta_1, \beta_2) \geq P(d_2, d_3)$ implies that the integers are in increasing order. The condition for the dimension of $M_2 \cap ( TM +M[T])$ is then
$$\max(\delta_1 + \delta_2 + \delta_3 - \beta_1, d_1 + d_2 + \delta_1 + \delta_3 - \beta_1 - d_1  ) \leq \dim ( M_2 \cap (TM +M[T])) \leq \min (d_1 + d_2,\delta_1 + \delta_3)$$
Note that $\dim ( M_2 \cap (TM +M[T]))$ must be equal to $\dim (M_2 \cap M[T]) + \delta_3$. If we denote by $\alpha_1$ the quantity $\dim (M_2 \cap M[T])$, the conditions are then
$$\max(\delta_1 + \delta_2 - \beta_1,  d_2 + \delta_1  - \beta_1  ) \leq \alpha_1 \leq \min (d_1 + d_2 - \delta_3,\delta_1)$$
The inequality $\alpha_1 \leq \min (d_1 + d_2 - \delta_3,\delta_1)$ is automatically satisfied if one has $P(\delta_1, \delta_2, \delta_3) \geq P(\alpha_1, \alpha_2, d_3)$ and $\alpha_1 + \alpha_2 = d_1 + d_2$. The other inequality gives the condition in the definition of $Y^{adm}$. \\
To finish the proof, one only checks that the condition $\beta_2 \leq \delta_2 \leq \beta_1$ is implied by the definition of $Y^{adm}$.
\end{proof}

\section{Hodge stratification}

\subsection{Unitary Case}

Let $F_0$ be a totally real field, and $F/F_0$ be a CM extension. Let $\Sigma$ be the set of embeddings of $F_0$ into $\overline{\mathbb{Q}_p}$; for each $\sigma \in \Sigma$, let $a_\sigma, b_\sigma$ be two integers such that the quantity $h :=a_\sigma + b_\sigma$ does not depend on $\sigma$. Let $k_0$ be a finite field containing all the residue fields of $F$.

\begin{defi}
Let $X$ be the PR variety over $k_0$ associated to $F/F_0$, with signature $(a_\sigma, b_\sigma)$.
\end{defi}

We refer to \cite{BH_PR} section 2 for the precise definition of $X$, but let us explain the main point of this variety. One has an universal abelian scheme $A$ over $X$, endowed with an action of $O_F$ (the ring of integers of $F$), a polarization (with some compatibility with the action of $O_F$). Let us now describe the PR datum. One has a sheaf $\omega_0 := \omega_{A}$ on $X$; it is locally free of rank $h [F_0 : \mathbb{Q}]$ and has an action of $O_F$. One has a decomposition $\Sigma = \coprod_{\pi |p} \Sigma_\pi$, where $\pi$ runs through the places of $F_0$ over $p$, and $\Sigma_\pi$ is the subset if embeddings above $\pi$. The sheaf $\omega_0$ decomposes as $\omega_0 = \oplus_{\pi | p } \omega_\pi$. One then distinguish several cases. \\
$ $\
Let $\pi$ be a prime of $F_0$ above $p$ and assume that $\pi$ splits as $\pi^+ \pi^-$ in $F$. Then the sheaf $\omega_\pi$ decomposes as $\omega_\pi^+ \oplus \omega_\pi^-$. The sheaf $\omega_\pi^+$ is locally free of rank $\sum_{\sigma \in \Sigma_\pi} a_\sigma$. Let $F_{0,\pi}$ be the completion of $F_0$ at $\pi$, $e$ the ramification index and $f$ the residual degree. Let $F_{0,\pi}^{ur}$ be the maximal unramified extension inside $F_{0,\pi}$, and let $\Sigma_\pi^{ur}$ be the set of embeddings of $F_{0,\pi}^{ur}$. One has the decomposition
$$\omega_\pi^+ = \oplus_{\sigma \in \Sigma_\pi^{ur}} \omega_{\pi, \sigma}^+$$
Let us fix an embedding $\sigma \in \Sigma_\pi^{ur}$, and let us consider the sheaf $\omega := \omega_{\pi, \sigma}^+$. It is locally free of rank $\sum a_\tau$, where the sum runs through the embeddings $\tau$ of $F_{0,\pi}$ extending $\sigma$. The PR datum thus consists in a filtration
$$0 \subseteq \omega_1 \subseteq \dots \subseteq \omega_e = \omega$$
where each graded part locally free of rank $a_i$. Note that this needs an ordering on the $a_\tau$. \\
$ $\\
Now let $\pi$ be a place above $p$ in $F_0$, and assume that $\pi$ is inert in $F$. Let $F_{0,\pi}$ be the completion of $F_0$ at $\pi$, $e$ the ramification index and $f$ the residual degree. Let $F_{0,\pi}^{ur}$ be the maximal unramified extension inside $F_{0,\pi}$, and let $\Sigma_\pi^{ur}$ be the set of embeddings of $F_{0,\pi}^{ur}$. Let $F_\pi$ be the completion of $F$ at $\pi$ and $F_{\pi}^{ur}$ be the maximal unramified extension inside $F_{pi}$. The sheaf $\omega_\pi$ decomposes as $\sum_{\sigma \in \Sigma_\pi^{ur}} \omega_{\pi, \sigma}$. The sheaf $\omega_{\pi, \sigma}$ is locally free of rank $eh$, and decomposes as $\omega_{\pi, \sigma, 1} \oplus \omega_{\pi, \sigma, 2}$, according to the action of the ring of integers of $F_\pi^{ur}$. Let us write $\omega := \omega_{\pi, \sigma, 1}$, it is locally free of rank $\sum a_\tau$, where the sum runs through the embeddings $\tau$ of $F_{0,\pi}$ extending $\sigma$. Similarly as in the previous case, one has a PR datum for this sheaf. \\
$ $\\
The case where the prime $\pi$ ramifies in $F_0$ can be dealt in a slightly more involved matter. Since we will not consider this case, we do not give any more details.

\begin{hypo}
We considers a prime $\pi$ of $F_0$ which does not ramified in $F$. We also suppose that the ramification index of $\pi$ is $3$.
\end{hypo}

Let us fix an embedding $\sigma \in \Sigma_\pi^{ur}$. One has thus a sheaf $\omega$ locally free of rank $a_1 + a_2 + a_3$, with a PR filtration
$$\omega_1 \subseteq \omega_2\subseteq \omega$$
with $\omega_1$ locally free of rank $a_1$,$\omega_2 / \omega_1$ locally free of rank $a_2$, and $\omega / \omega_2$ is locally free of rank $a_3$.

Let $\mathcal{E}$ be the $\sigma$ part of the De Rham cohomology. It is locally free of rank $h$ over $\mathcal{O}_X [T] / T^e$, and by the Hodge filtration it has $\omega$ as locally a direct factor. \\

\begin{defi}
Let $k$ be a field in characteristic $p$, and $x \in X(k)$. The above datum defines a unique element $Hdg(x) \in Y^{adm}$.
\end{defi}

The Hodge stratification (attached to $\sigma$) is then 
$$X = \coprod_{y \in Y^{adm}} X_y$$
where $X_y$ consists of all the points of $X$ with $Hdg(x) = y$.

Let us define an order on $Y$.

\begin{defi}
Let $y = (P_1, P_2, P_3)$ and $y' = (P_1', P_2', P_3')$ be elements of $Y$. One says that $y \geq y'$ if $P_i \geq P_i'$ for all $i \in \{1,2,3 \}$.
\end{defi}

\begin{theo}
Let $y \in Y^{adm}$. Then
$$ \overline{X_y} = \coprod_{y' \geq y} X_{y'} $$
\end{theo}

\begin{proof}
The Hodge polygon goes up by specialization (see \cite{Ka}). This proves the inclusion 
$$ \overline{X_y} \subseteq \coprod_{y' \geq y} X_{y'} $$
Now, let $y' \geq y$ and $x \in X_{y'} (k)$. We want to prove that there exists a deformation of $x$ which is in $X_y$. \\
Let us write $y' = (P_1', P_2', P_3')$, with $P_1' = P(\delta_1', \delta_2', \delta_3')$, $P_2 = P(\alpha_1', \alpha_2')$, $P_3 = P(\beta_1', \beta_2')$ ($\delta_1' \geq \delta_2' \geq \delta_3'$, $\alpha_1' \geq \alpha_2'$, $\beta_1' \geq \beta_2'$), and similarly $y = (P_1, P_2, P_3)$. By Serre-Tate and Grothendieck-Messing (see \cite{BBM}), it is enough to lift the Hodge filtration. Let $R = k[[X]]$, and let $\mathcal{E}_R := \mathcal{E} \otimes R$. One lifts $\omega$ to  direct summand $\omega_R \subseteq \mathcal{E}_R$ such that $\omega_R \otimes_R k((X))$ will be a PR datum of type $y$. \\
One lifts successively the filtration $0 \subseteq \omega_1 \subseteq \omega_2 \subseteq \omega$. First, one lifts $\omega_1$ to $\omega_{1, R}$ a free $R$-module of rank $d_1$ inside $\mathcal{E}_R [T]$. Then one considers the filtration
$$ \omega_{1,R} \subseteq \mathcal{E}_R [T] \subseteq T^{-1} \omega_{1,R}$$
One lifts $\omega_{2,R}$ inside $T^{-1} \omega_{1,R}$, such that it contains $ \omega_{1,R}$, and the intersection with $\mathcal{E}_R [T]$ has dimension $\alpha_1$ in generic fiber. This is possible since the intersection of $\mathcal{E} [T] $ and $\omega_2$ has dimension $\alpha_1' \geq \alpha_1$. Let $\mathcal{F}_{k((X))} := 
\mathcal{E}_{k((X))} [T] + (\omega_{2,R} \otimes_R k((X)))$, and $\mathcal{G}_{k((X))} := \mathcal{E}_{k((X))} [T^2] \cap (T^{-1} (\omega_{2,R} \otimes_R k((X))))$. These are $k((X))$ vector spaces of dimension $h+ \alpha_2$ and $h + \alpha_1$ respectively. Let $\mathcal{F} := \mathcal{F}_{k((X))} \cap \mathcal{E}_R$, and $\mathcal{G} := \mathcal{G}_{k((X))} \cap \mathcal{E}_R$. We will assume that the lift of $\omega_2$ is done in such a way that the intersection of $\omega$ with the reduction of $\mathcal{F}$ and $\mathcal{G}$ have maximal dimension. These dimensions can be made equal to 
$$ \min (\delta_1' + \alpha_2, \beta_1' + d_1) \qquad \min (\delta_1' + \delta_2', \alpha_1 + \beta_1')$$
respectively. One then considers the filtration
$$\omega_{2,R} \subseteq \mathcal{F} \subseteq T^{-1} \omega_{1,R} \subseteq \mathcal{G} \subseteq T^{-1} \omega_{2,R}$$
One lifts $\omega$ to $\omega_R$ inside $T^{-1} \omega_{2,R}$ and containing $\omega_{2,R}$ such that the intersection with $\mathcal{F}$, $T^{-1} \omega_{1,R}, \mathcal{G}$ have dimension $\delta_1 + \alpha_2, \beta_1 + d_2, \delta_1 + \delta_2$ respectively in generic fiber. This is indeed possible since one has
$$\delta_1 + \alpha_2 \leq \min (\delta_1' + \alpha_2, \beta_1' + d_1)$$
$$ \beta_1 + d_2 \leq \beta_1'+d_1$$
$$\delta_1 + \delta_2 \leq \min (\delta_1' + \delta_2', \alpha_1 + \beta_1')$$
To achieve this, one uses the following lemma.
\end{proof}

\begin{lemm} \label{lemmfilt}
Let $R = k[[X]]$, and consider $M$ a free $R$ module of rank $h$. Assume that one has direct summands
$$0 \subseteq M_1 \subseteq M_2 \subseteq \dots \subseteq M_l = M$$
such that $M_i$ is free of rank $h_i$. Let $\overline{M_i} :=M_i \otimes_R k$, and let $\overline{L}$ be a vector subspace of $\overline{M}$, and let $d_i := \dim \overline{L} \cap \overline{M_i}$.  One has automatically $0 \leq d_{i+1} - d_i \leq h_{i+1} - h_i$. \\
Let $d_1', \dots d_l'$ be integers, with $d_l' = d_l$, $0 \leq d_{i+1}' - d_i' \leq h_{i+1} - h_i$, and $d_i' \leq d_i$ for $1 \leq i \leq l-1$. Then there exists a lift $L \subseteq M$ of $\overline{L}$, which is a direct summand, such that the intersection of $L$ with $M_i$ has dimension $d_i'$ in generic fiber. 
\end{lemm}

\begin{proof}
One constructs a basis $f_1, \dots, f_{d_l}$ for the lift $L$ inductively. \\
The space $\overline{L} \cap \overline{M_1}$ has dimension $d_1'$, let $e_1, \dots, e_{d_1}$ be a family of $M_1$ such that its reduction is a basis for $\overline{L} \cap \overline{M_1}$. We set $f_i = e_i$ for $1 \leq i \leq d_1'$, and $f_i = e_i + X v_{i-d_1'}$ for $d_1'+ 1 \leq i \leq d_1$, where the elements $v_1, \dots, v_{d_1'-d_1}$ are vectors yet to be determined. \\
Next, one distinguish two cases. First assume that $d_2' - d_1' \leq d_2 - d_1$. Let $e_{d_1+1}, \dots, e_{d_2}$ be a family of $M_2$, such that the reduction is a basis for $\overline{L} \cap \overline{M_2} / \overline{L} \cap \overline{M_1}$. We set $f_i = e_i$ for $d_1 +1 \leq i \leq d_1 + d_2' - d_1'$, and $f_i = e_i + X w_{i+d_1' - d_1 - d_2'}$ for $d_1 + d_2' - d_1' +1  \leq i \leq d_2 $, where the vectors $w_i$ are yet to be determined. In total, there are thus $d_2 - d_2'$ vectors to be determined. There is then $d_2 - d_2'$ vectors to determine in total. \\
Now assume that $d_2' - d_1' > d_2 - d_1$. We set $f_i = e_i$ for $d_1 + 1 \leq i \leq d_2$. One also complete the family $e_{d_1 + 1}, \dots, e_{d_2}$ into $e_{d_1 +1}, \dots, e_{d_1 + h_2 - h_1}$, which is a basis for $M_2 / M_1$. One then sets $v_i = e_{d_2 + i}$ for $1 \leq i \leq (d_2' - d_1') - (d_2 - d_1)$. This is possible since $d_2' - d_1' \leq h_2 - h_1$. After that step, there is only $d_2 - d_2'$ vectors to determine. \\
One repeats this process. Since $d_l = d_l'$, at the end there is no more vectors to determine, and this is how one constructs the lift $L$.
\end{proof}

\section{Hilbert-Siegel variety}

In this section $F$ denotes a totally real field. We denote by $X$ the Hilbert-Siegel variety attached to $F$ with PR condition (see \cite{BH_PR} section 2.4). Let us fix a prime $\pi$ above $p$ in $F$, such that the ramification index is $3$. Let us fix an unramified embedding $\sigma$. We consider the sheaf $\omega_\sigma$ on $X$. It is locally free of rank $3g$, and has an action of $k[T]/T^3$. It has a PR datum
$$ \omega_1 \subseteq \omega_2 \subseteq \omega$$
such that each graded part is locally free of rank $g$. Ine also consider the $\sigma$ part of de Rham cohomology $\mathcal{E}$, which is locally free of rank $6g$. It is also equipped with a pairing induced by the polarization, and $\omega \subseteq \mathcal{E}$ is totally isotropic for this pairing. The PR datum is compatible with the pairing in the sense that
$$\omega_1^\bot = T^{-2} \omega_1 \qquad \omega_2^\bot = T^{-1} \omega_2$$
these equalities taking place in $\mathcal{E}$. 

\begin{defi}
We define the subset $Y^{pol} \subseteq Y^{adm}$ as the subset consisting of the points $(P_1, P_2, P_3) \in Y^{adm}$, with
$$P_1 = P(r, g, 2g-r)$$
for some integer $g \leq r \leq 2g$. 
\end{defi}

If $x \in X(k)$, the the PR datum will land in $Y^{pol}$. This is because the polarization induces an isomrophism $\omega \simeq (\mathcal{E} / \omega)^\vee$.

\begin{defi}
Let $k$ be a field in characteristic $p$, and $x \in X(k)$. The above datum defines a unique element $Hdg(x) \in Y^{pol}$.
\end{defi}

The Hodge stratification (attached to $\sigma$) is then 
$$X = \coprod_{y \in Y^{pol}} X_y$$
where $X_y$ consists of all the points of $X$ with $Hdg(x) = y$.

\begin{theo}
Let $y \in Y^{pol}$. Then
$$ \overline{X_y} = \coprod_{y' \geq y} X_{y'} $$
\end{theo}

\begin{proof}
Since the Hodge polygon goes up by specialization, one has the inclusion 
$$ \overline{X_y} \subseteq \coprod_{y' \geq y} X_{y'} $$
Now, let $y' \geq y$ and $x \in X_{y'} (k)$. We want to prove that there exist a deformation of $x$ which is in $X_y$. \\
Let us write $y' = (P_1', P_2', P_3')$, with $P_1' = P(\delta_1', g, 2g - \delta_1')$, $P_2 = P(\alpha_1', \alpha_2')$, $P_3 = P(\beta_1', \beta_2')$ ($2g \geq \delta_1' \geq g$, $\alpha_1' \geq \alpha_2'$, $\beta_1' \geq \beta_2'$), and similarly $y = (P_1, P_2, P_3)$. By Serre-Tate and Grothendieck-Messing, it is enough to lift the Hodge filtration. Let $R = k[[X]]$, and let $\mathcal{E}_R := \mathcal{E} \otimes R$. One lifts $\omega$ to  direct summand $\omega_R \subseteq \mathcal{E}_R$, totally isotropic, such that $\omega_R \otimes_R k((X))$ will be a PR datum of type $y$. \\
One lifts successively the filtration $0 \subseteq \omega_1 \subseteq \omega_2 \subseteq \omega$. First, one lifts $\omega_1$ to $\omega_{1, R}$ a free $R$-module of rank $g$ inside $\mathcal{E}_R [T]$, totally isotropic for the induced pairing. Then one considers the filtration
$$ \omega_{1,R} \subseteq \mathcal{E}_R [T] \subseteq T^{-1} \omega_{1,R}$$
One lifts $\omega_{2,R}$ inside $T^{-1} \omega_{1,R}$, such that it contains $ \omega_{1,R}$, and the intersection with $\mathcal{E}_R [T]$ has dimension $\alpha_1$ in generic fiber. One also requires that $\omega_{2,R} / \omega_{1,R}$ is totally isotropic for the induced pairing on $T^{-1} \omega_{1,R} / \omega_{1,R}$. This is possible since the intersection of $\mathcal{E} [T] $ and $\omega_2$ has dimension $\alpha_1' \geq \alpha_1$. Let $\mathcal{F}_{k((X))} := 
\mathcal{E}_{k((X))} [T] + (\omega_{2,R} \otimes_R k((X)))$, and $\mathcal{G}_{k((X))} := \mathcal{E}_{k((X))} [T^2] \cap (T^{-1} (\omega_{2,R} \otimes_R k((X))))$. These are $k((X))$ vector spaces of dimension $2g+ \alpha_2$ and $2g + \alpha_1$ respectively, and they are the orthogonal of one another. Let $\mathcal{F} := \mathcal{F}_{k((X))} \cap \mathcal{E}_R$, and $\mathcal{G} := \mathcal{G}_{k((X))} \cap \mathcal{E}_R$. We will assume that the lift of $\omega_2$ is done in such a way that the intersection of $\omega$ with the reduction of $\mathcal{F}$ and $\mathcal{G}$ have maximal dimension. These dimensions can be made equal to 
$$ \min (\delta_1' + \alpha_2, \beta_1' + g) \qquad \min (\delta_1' + g, \alpha_1 + \beta_1')$$
respectively. One then considers the filtration
$$\omega_{2,R} \subseteq \mathcal{F} \subseteq T^{-1} \omega_{1,R} \subseteq \mathcal{G} \subseteq T^{-1} \omega_{2,R}$$
One lifts $\omega$ to $\omega_R$ inside $T^{-1} \omega_{2,R}$ and containing $\omega_{2,R}$ such that the intersection with $\mathcal{F}$, $T^{-1} \omega_{1,R}, \mathcal{G}$ have dimension $\delta_1 + \alpha_2, \beta_1 + g, \delta_1 + g$ respectively in generic fiber. This is indeed possible since one has
$$\delta_1 + \alpha_2 \leq \min (\delta_1' + \alpha_2, \beta_1' + g)$$
$$ \beta_1 + g \leq \beta_1'+g$$
$$\delta_1 + g \leq \min (\delta_1' + g, \alpha_1 + \beta_1')$$
One also requires that $\omega_R / \omega_{2,R}$ is totally isotropic for the induced pairing on $ T^{-1} \omega_{2,R} / \omega_{2,R}$. All of this is achieved with the following lemma.
\end{proof}

\begin{lemm}
Let $R = k[[X]]$, and consider $M$ a free $R$ module of rank $2g$ with a perfect pairing. Assume that one has direct summands
$$0 \subseteq M_1 \subseteq M_2 \subseteq \dots \subseteq M_l = M$$
such that $M_i$ is free of rank $h_i$, with $M_i^\bot = M_{l-i}$. Let $\overline{M_i} :=M_i \otimes_R k$, and let $\overline{L}$ be a maximal totally isotropic vector subspace of $\overline{M}$, and let $d_i := \dim \overline{L} \cap \overline{M_i}$.  One has automatically $0 \leq d_{i+1} - d_i \leq h_{i+1} - h_i$, $h_{l-i} =2g - h_i$ and $d_{l-i} =g - h_i + d_i$, $d_l =g$. \\
Let $d_1', \dots d_l'$ be integers, with $d_l' = g$, $0 \leq d_{i+1}' - d_i' \leq h_{i+1} - h_i$, $d_{l-i}' = g -h_i +d_i'$ and $d_i' \leq d_i$ for $1 \leq i \leq l-1$. Then there exist a lift $L \subseteq M$ of $\overline{L}$, which is a totally isotropic direct summand, such that the intersection of $L$ with $M_i$ has dimension $d_i'$ in generic fiber. 
\end{lemm}

\begin{proof}
This is a variant of the proof of the lemma \ref{lemmfilt} from the previous section. Let $k$ be the integer part of $l/2$. It is enough to consider the modules $M_1, \dots, M_k$, which are moreover totally isotropic. One then applies process described in the previous section. One constructs a module $L_1$ of rank $d_1'$, such that $L_1$ is included in $M_1$, and the reduction of $L_1$ is in $\overline{L}$. One has tried to look for a lift of $\overline{L} \cap \overline{M_1}$, but there are $d_1 - d_1'$ vectors yet to determine. One then considers the module $L_1 ^\bot$, and will lift all the vectors inside this module. We then consider the vector space $\overline{M_2} \cap \overline{L}$. One will look for a (partial) lift $L_2$ of rank $d_2'$, which will be totally isotropic. One then repeats this process to construct the lift $L$ of $\overline{L}$. 
\end{proof}

\bibliographystyle{amsalpha}

\end{document}